\documentclass[10pt,a4paper]{article}

\usepackage{amsfonts}
\usepackage{amsmath}
\usepackage{amstext,amssymb,amsthm}
\usepackage[pdftex]{hyperref}

\newtheorem{theorem}{Theorem}[section]

\newtheorem{lemma}[theorem]{Lemma}
\newtheorem{proposition}[theorem]{Proposition}
\theoremstyle{definition}
\newtheorem{definition}[theorem]{Definition}

\newtheorem{remark}[theorem]{Remark}

\title{{\bf An Osserman-type condition on $g.f.f$-manifolds with Lorentz metric}
\footnote{The author wishes to express her thanks to professors A.M. Pastore and M. Falcitelli for helpful comments and for many stimulating conversations. The work was supported by the Research Program n.\ 01.08 of University of Bari}}
\author{Letizia Brunetti}
\date{}

\begin{document}

\thispagestyle{empty}

\maketitle

\begin{abstract}
A condition of Osserman type, called $\varphi$-null Osserman condition, is introduced and studied in the context of Lorentz globally framed $f$-manifolds. An explicit example shows the naturalness of this condition in the setting of Lorentz $\mathcal{S}$-manifolds. We prove that a Lorentz $\mathcal{S}$-manifold with constant $\varphi$-sectional curvature is $\varphi$-null Osserman, extending a result stated for Lorentz Sasaki space forms. Then we state a characterization for a particular class of $\varphi$-null Osserman $\cal{S}$-manifolds. Finally, some examples are examined.
\end{abstract}

\section{Introduction}
The study of the behaviour of the Jacobi operators is an important topic in Riemannian and, more generally, in semi-Riemannian geometry. More precisely, let $(M,g)$ be a Riemannian manifold with curvature tensor $R$ and consider a point $p$ in $M$. For any unit vector $X\in T_p M$, the symmetric endomorphism $R_X=R_p(\cdot,X)X:X^\bot\rightarrow X^\bot$ is called the Jacobi operator with respect to $X$. If the eigenvalues of $R_X$ are independent of the choices of $X$ and $p$, one says that $(M,g)$ is an Osserman manifold (\cite{GKV}).

Several results have been obtained looking for the solution of the Osserman Conjecture (\cite{C0,Os}), which states that an Osserman manifold is flat or it is locally a rank-one symmetric space (\cite{C0,C1,C2,N01,N02,N03}). Osserman manifolds have been studied in the Lorentzian context (\cite{BBG,GKV0,GKV1}), where a complete solution for the Osserman conjecture has been found. Recently, in \cite{AD}, Atindogbe and Duggal have introduced and studied suitable operators of Jacobi type associated with a semi-Riemannian degenerate metric. 

In (\cite{GKV1}) the authors defined the Jacobi operator $\bar{R}_u$, $u$ being a null (or lightlike) vector tangent to a Lorentz manifold $M$. Given a unit timelike vector $z$ tangent to $M$, they introduced and investigated the so-called null Osserman condition with respect to $z$ (see also \cite{GKV}).

Obviously, Lorentz almost contact manifolds are studied in this context. In particular, a Lorentz Sasaki space form, whose characteristic vector field $\xi$ is timelike, is globally null Osserman with respect to $\xi$ (\cite{GKV}). This result does not hold in the context of Lorentz globally framed $f$-manifolds $(M^{2n+s},\varphi,\xi_\alpha,\eta^\alpha,g)$, $s\geq 2$, as we will see with a counterexample.

This motivates the introduction of a more general condition of Osserman type, which we will call \emph{$\varphi$-null Osserman condition}.

The main results of this paper state the links between the $\varphi$-null Osserman condition and the behaviour of the $\varphi$-sectional curvature in Lorentz $\mathcal{S}$-manifolds. After a preliminary section, where we gather some facts about $g.f.f$-manifolds, needed in the rest of the paper, in Section 3 we discuss the relationship between the null Osserman condition and the Lorentz $\mathcal{S}$-structures, giving an example of Lorentz $\mathcal{S}$-space form which does not satisfy the null Osserman conditions. We endow the compact Lie group $U(2)$ with a Lorentz $\cal{S}$-structure of rank $2$. This manifold is an $\mathcal{S}$-space form with two characteristic vector fields $\xi_1$ and $\xi_2$, $\xi_1$ timelike, that does not satisfy the null Osserman condition with respect to $\xi_1$.

In Section 4 we introduce the notion of $\varphi$-null Osserman manifold, and we state that a Lorentz $\cal{S}$-manifold with constant $\varphi$-sectional curvature is $\varphi$-null Osserman with respect to the timelike characteristic vector field. We prove, in Section 5, an algebraic characterization for the Riemannian curvature tensor field in a particular class of $\varphi$-null Osserman Lorentz $\mathcal{S}$-manifolds. Namely, we divide this section in two parts. In the first subsection we deal with technical results which are very useful in the second subsection where we state the main result. Moreover, we look at the behaviour of the $\varphi$-sectional curvature when the number of the eigenvalue of the Jacobi operator is one.

In particular, it is interesting to note that the existence of the only eigenvalue $1$ of the Jacobi operator is related to the $\varphi$-sectional flatness of the manifold.  

Finally in the case of $4$-dimensional $\varphi$-null Osserman manifolds we find a compact example, using the Lie group $U(2)$, and also a non compact example. 
   
All manifolds, tensor fields and maps are assumed to be smooth, moreover we suppose all manifolds are connected. We will use the Einstein convention omitting the sum symbol for repeated indexes. Following the notations of S. Kobayashi and K. Nomizu (\cite{KN}), for the curvature tensor
$R$ we have $R(X,Y)Z = \nabla_X\nabla_Y Z - \nabla_Y \nabla_X Z - \nabla_{[X,Y ]}Z$, and $R(X,Y,Z,W) =
g(R(Z,W)Y,X)$, for any $X,Y,Z,W\in\mathfrak{X}(M)$.
The \emph{sectional curvature} $K_{p}(\pi)$ at $p$ of a non-degenerate $2$-plane $\pi=span\{  X,Y\}  $ is given by
\[
K_{p}(  \pi)=K_p(X,Y)  =\frac{{R}_{p}(  X,Y,X,Y)  }%
{\Delta(\pi)}=\frac{{g}_{p}(  {R}_{p}(
X,Y)Y  ,X)  }{\Delta(\pi) } , \]
where $\Delta(\pi)=g(X,X)g(Y,Y)-g(X,Y)^2\neq 0$.

\section{Preliminaries}
Following \cite{Bl1,LP,T}, we recall some definitions. An almost contact manifold is a $(2n+1)$-dimensional manifold $M$ endowed with an almost contact structure, i.e. $M^{2n+1}$ has a $(1,1)$-tensor field $f$ such that $rank(f)=2n$, a $1$-form $\eta$ and a vector field $\xi$ satisfying $f^2(X) =-X +\eta(X)\xi$ and $\eta(\xi) = 1$. Moreover, if $g$ is a semi-Riemannian metric on $M^{2n+1}$ such that, for any $X,Y\in\mathfrak{X}(M^{2n+1})$,
$$g(fX,fY) = g(X,Y)-\varepsilon\eta(X)\eta(Y),$$
where $\varepsilon =\pm1$ according to the causal character of $\xi$, then $M^{2n+1}$ is called an indefinite almost contact manifold. Such a manifold is said to be an indefinite contact manifold if $d\eta=\Phi$, where $\Phi$ is defined by $\Phi(X,Y)=g(X,fY)$. Furthermore, if the structure $(f,\xi,\eta)$ is normal, that is $N=[f,f]+2d\eta\otimes\xi=0$, then the indefinite contact structure is called an indefinite Sasaki structure and, in this case, the manifold $(M^{2n+1},f,\xi,\eta,g)$ is called indefinite Sasaki.

In the Riemannian case a generalization of these structures have been studied by Blair in \cite{Bl}, by Goldberg and Yano in \cite{GY}. In \cite{LP} we studied such structures in semi-Riemannian context. 
  
A manifold $M$ is called a globally framed $f$-manifold (briefly $g.f.f$-manifold) if it is endowed with a nowhere-vanishing $(1,1)$-tensor field $\varphi$ of constant rank, such that $\ker \varphi$ is parallelizable i.e. there exist global vector fields  $\xi_\alpha$, $\alpha\in\{1,\ldots,s\}$, and $1$-forms $\eta^\alpha$, satisfying 
\[
\varphi^{2}=-I+\eta^\alpha\otimes\xi_\alpha \text{ and } \eta^\alpha(\xi_\beta)=\delta_\beta^\alpha.
\]

A $g.f.f$-manifold $(M^{2n+s},\varphi,\xi_\alpha,\eta^\alpha)$, $\alpha\in \{1,\ldots,s\}$, is said to be an indefinite $g.f.f$-manifold if $g$ is a semi-Riemannian metric satisfying the following compatibility condition
\begin{equation*}
g(\varphi X,\varphi Y)=g(X,Y)-\varepsilon_\alpha\eta^\alpha(X)\eta^\alpha(Y)
\end{equation*}
for any vector fields $X, Y$,  being $\varepsilon_\alpha=\pm 1$ according to whether $\xi_\alpha$ is
spacelike or timelike. Then, for any $\alpha\in\{1,\ldots,s\}$ and $X\in\mathfrak{X}(M^{2n+s})$, one has $\eta^\alpha(X)=\varepsilon_\alpha g(X,\xi_\alpha)$.

An indefinite $g.f.f$-manifold is an indefinite $\mathcal{S}$-manifold if it is normal and $d\eta^\alpha=\Phi$, for any $\alpha\in\{1,\ldots,s\}$, where $\Phi(X,Y)=g(X,\varphi Y)$ for any $X,Y\in \mathfrak{X}(M^{2n+s})$. The normality condition is expressed by the vanishing of the tensor field $N=N_\varphi+2d\eta^\alpha\otimes \xi_\alpha$, $N_{\varphi}$ being the Nijenhuis torsion of $\varphi$. 

\noindent Furthermore, as proved in \cite{LP}, the Levi-Civita connection of an indefinite $\mathcal{S}$-manifold satisfies:
\begin{equation*}%\label{LC}
(\nabla_X\varphi)Y= g(\varphi X,\varphi Y) \widetilde{\xi}+\widetilde{\eta}(Y)\varphi^{2}(X),
\end{equation*}
where $\widetilde{\xi} = \sum_{\alpha=1}^s \xi_\alpha$ and $\widetilde{\eta} =\varepsilon_\alpha\eta^\alpha$.
Note that, for $s=1$, we reobtain the notion of indefinite Sasaki manifold.

We recall that $\nabla_X \xi_\alpha=-\varepsilon_\alpha\varphi X$ and $\ker \varphi $ is an integrable flat distribution since $\nabla_{\xi_\alpha} \xi_\beta =0$, for any $\alpha,\beta\in\{1,\ldots,s\}$. Anyway, an indefinite $\mathcal{S}$-manifold is never flat and it is never a real space form since, for example,  $K(X,\xi_\alpha)=\varepsilon_\alpha$ for any non lightlike $X\in\operatorname{Im}\varphi_p$.

For more details we refer to \cite{LP}, where we describe three examples of non compact indefinite $\mathcal{S}$-manifolds. More precisely we construct two different indefinite $\mathcal{S}$-structures with metrics of index $\nu=2$ on $\mathbb{R}^6$ and an indefinite $\mathcal{S}$-structure with Lorentz metric on $\mathbb{R}^4$. Moreover, in \cite{LP2} we give explicit examples of compact indefinite $g.f.f$-manifolds and indefinite $\mathcal{S}$-manifolds.

We also remark that every $g.f.f$-manifold is subject to the following topological condition: it has to be either non compact or compact with vanishing Euler characteristic, since it admits never vanishing vector fields. This implies that such a $g.f.f$-manifold always admits Lorentz metrics.

Let us fix few notation about curvature tensor field. As usual, a $2$-plane $\pi=span\{ X,{\varphi}X\}$ in $T_{p}{M}$, with $p\in{M}$ and $X\in\operatorname{Im}\varphi_{p}$, is said to be a ${\varphi}$-\emph{plane} and the sectional curvature at $p$ of such a plane, with $X$ a non lightlike vector, is called the ${\varphi}$-\emph{sectional curvature} at $p$ and is denoted by $H_{p}(X)$.

An indefinite $\mathcal{S}$-manifold $(M,\varphi,\xi_{\alpha},\eta^{\alpha},g)$ is said to be an indefinite $\cal{S}$-space form if the ${\varphi}$-sectional curvature $H_{p}(X)$ is constant, for any point and any ${\varphi}$-plane. In particular, in \cite{LP} it is proved that an indefinite $\mathcal{S}$-manifold $(M,\varphi,\xi_{\alpha},\eta^{\alpha},g)$ is  an indefinite $\cal{S}$-space form with $H_{p}(X)=c$ if and only if the Riemannian
$(0,4)$-type curvature tensor field ${R}$ is given by
\begin{align}
	{R}(X,Y,Z,W)&=-\frac{c+3\varepsilon}{4}\{{g}({\varphi}Y,{\varphi}Z){g}({\varphi}X,{\varphi}W)-{g}(  {\varphi}X,{\varphi}Z){g}({\varphi}Y,{\varphi}W)\} \label{equivalente}\\
&\quad-\frac{c-\varepsilon}{4}\{\Phi(W,X)
\Phi(Z,Y)-\Phi(Z,X)\Phi(W,Y)  \nonumber\\
&  \quad+2\Phi(X,Y)\Phi(W,Z)
\} -\{\widetilde{\eta}(W)\widetilde{\eta}(X){g}({\varphi}Z,{\varphi}Y) \nonumber\\
&  \quad  -\widetilde{\eta}(W)\widetilde{\eta}(Y)
{g}({\varphi}Z,{\varphi}X)+\widetilde{\eta}(Y)\widetilde{\eta}(Z){g}({\varphi}W,{\varphi}X)\nonumber\\
&  \quad-\widetilde{\eta}(Z)\widetilde{\eta}(X)  {g}({\varphi}W,{\varphi}Y)\},\nonumber
\end{align}
for any vector fields $X$, $Y$, $Z$ and $W$ on $M$, where $\varepsilon=\sum_{\alpha=1}^s \varepsilon_\alpha$.

In regard to the curvature tensor of an indefinite $\mathcal{S}$-manifold, it is important to recall the following formulas, for any $X,Y,Z,W\in\operatorname{Im}\varphi$ and any $\alpha,\beta,\gamma,\delta\in\{1,\ldots,s\}$:
\begin{align}
	&R(X,\xi_\alpha,X,Y)=\varepsilon_\alpha g(X,X)g(\widetilde{\xi},Y)=0,\nonumber\\
	&R(\xi_\alpha,X,\xi_\beta,Y)=\varepsilon_\alpha\varepsilon_\beta g(X,Y),\nonumber\\
	&R(\xi_\alpha,X,\xi_\beta,\xi_\gamma)=\varepsilon_\alpha\varepsilon_\beta g(X,\xi_\gamma)=0,\label{PropS}\\
	&R(\xi_\alpha,\xi_\delta,\xi_\beta,\xi_\gamma)=0,\nonumber\\
	&R(X,Y,\varphi Z, W)+R(X,Y,Z,\varphi W)=\varepsilon P(X,Y;Z,W),\nonumber
\end{align}
where $P(X,Y;Z,W)=\Phi(X,Z)g(Y,W)-\Phi(X,W)g(Y,Z)-\Phi(Y,Z)g(X,W)+\Phi(Y,W)g(X,Z)$.

Finally, we recall some useful properties for a curvature-like algebraic tensor. Let $(V,g)$ be a pseudo-Euclidean real vector space of index $\nu$, $0<\nu<\dim V$. A multilinear map $F:V^4\rightarrow \mathbb{R}$ is called a curvature-like map (or curvature-like algebraic tensor) if it satisfies the following conditions
\begin{align*}
	F(y,x,z,w)&=-F(x,y,z,w),\\
	F(z,w,x,y)&=F(x,y,z,w),\\
	F(x,y,z,w)&+F(x,z,w,y)+F(x,w,y,z)=0.
\end{align*}
For any non-degenerate $2$-plane $\pi=span\{z,w\}$ in $V$ it is possible to define the number
\[
k(z,w)=\frac{F(z,w,z,w)}{\Delta(\pi)}.
\]
If $k(z,w)$ is constant for any non-degenerate $2$-plane and $k(z,w)=k$ then one gets $F(x,y,z,w)=k \left(g(x,z)g(y,w)-g(y,z)g(x,w)\right)$.
Now, arguments similar to those in Proposition 28 (\cite[page 229]{O'N}), can be used to prove the following result. 
\begin{lemma}\label{2pianidegeneri}
Let $(V,g)$ be a Lorentz real vector space and $F:V^4\rightarrow \mathbb{R}$ a curvature-like map. Then the following conditions are equivalent.
\begin{itemize}
	\item [a)]$F(x,y,z,w)=k \left(g(x,z)g(y,w)-g(y,z)g(x,w)\right)$,
	\item [b)] $F(x,y,y,x)=0$ for any degenerate plane $\pi=span\{x,y\}$ in $V$.
\end{itemize}
\end{lemma}

\section{Null Osserman condition and Lorentz $\mathcal{S}$-manifolds}
It is well-known that a Lorentz manifold has constant sectional curvature at a point $p$ if and only if it satisfies the Osserman condition at $p$.

Contrary to this, no Lorentz $\mathcal{S}$-manifold can satisfy the Osserman condition since, as remarked in Section 2, a Lorentz $\mathcal{S}$-manifold can not have constant sectional curvature. 

In \cite{GKV1} the authors introduce another Osserman condition, named the null Osserman condition.
Namely, let $(M,g)$ be a Lorentz manifold, $p\in M$ and $u$ a null vector in $T_p M$. Then the orthogonal complement $u^\bot$ of $u$ is a degenerate vector space since $span\{u\}\subset u^\bot$. So one considers the quotient space $\bar{u}^\bot=u^\bot/span\{u\}$ and the canonical projection $\pi:u^\bot\rightarrow \bar{u}^\bot$. It is possible to define a positive definite inner product $\bar{g}$ on $\bar{u}^\bot$ putting
\[
\bar{g}(\bar{x},\bar{y})=g(x,y),
\]
where, for any $x,y\in u^\bot$, $\bar{x}=\pi(x)$ and $\bar{y}=\pi(y)$.

From now on, every bar-object will stand for geometrical objects related to $\bar{u}^\bot$. 
So, fixed a null vector $u\in T_p M$, the Jacobi operator with respect to $u$ can be defined by the linear map $\bar{R}_u: \bar{u}^\bot\rightarrow\bar{u}^\bot$ such that $\bar{R}_u \bar{x}=\pi(R(x,u)u)$ (\cite{GKV1} and Definition 3.2.1 in \cite{GKV}).

Clearly, $\bar{R}_u$ is self-adjoint with respect to $\bar{g}$, hence $\bar{R}_u$ is diagonalizable.

In Lorentzian geometry it is well-known that a null vector $u$ and a timelike vector $z$ are never orthogonal. Hence, in a Lorentz manifold $(M,g)$, the null congruence set determined by a timelike vector $z\in T_p M$ at $p$, denoted by $N(z)$, is defined  by 
\[
N(z)=\{u\in T_pM \; |\; g(u,u)=0, \;g(u,z)=-1\}.
\] 
A Lorentz manifold $(M,g)$ is called null Osserman with respect to a unit timelike vector $z\in T_p M$ at a point $p$ if the characteristic polynomial of $\bar{R}_u$ is independent of $u\in N(z)$. Let $L$ be a timelike line subbundle of $TM$. If $(M,g)$ is null Osserman with respect to each unit timelike vector $z\in L$, then $(M,g)$ is called pointwise null Osserman with respect to $L$. Moreover, if $(M,g)$ is pointwise null Osserman with respect to $L$ and the characteristic polynomial of $\bar{R}_u$ is independent of the choice of a unit $z\in L$, then $(M,g)$ is said to be globally null Osserman with respect to $L$.

Another set associated to a unit timelike vector $z$ in $T_p M$ is the celestial sphere $S(z)$ of $z$ given by
\[
S(z)=\{x\in z^\bot\ |\ g(x,x)=1\}.
\]

According to a result in \cite{GKV}, using the celestial sphere of $z$, one can obtain all the elements of $N(z)$. In fact one has
\[
\forall\, u\in N(z) \ \exists |\, x \in S(z) \ \text{such}\ \text{that}\ u=z+x.
\]
It is very natural to use this definition in the context of Lorentz contact manifolds. In particular, as stated in \cite{GKV},  Lorentz Sasaki space forms are globally null Osserman with respect to the timelike characteristic vector field. An easy example shows that in a Lorentz $\mathcal{S}$-space form the null Osserman condition with respect to a timelike characteristic vector does not hold.

Indeed, considering the $4$-dimensional manifold $U(2)$ and the Lie algebra $\mathfrak{u}(2)$, we denote by $\xi_1,\xi_2,X,Y$ the left-invariant vector fields on $U(2)$, determined, in the same order, by the basis $\{\imath E_{11}, -\imath E_{22}, E_{12}-E_{21}, \imath(E_{12}+E_{21})\}$ of $\mathfrak{u}(2)$, where $(E_{ij})_{i,j \in \{1,2\}}$ is the canonical basis of $gl(2,\mathbb C)$. Then, we get:
\[
[X,Y]=2\xi_1+2\xi_2,\quad [X,\xi_\alpha]=-Y, \quad [Y,\xi_\alpha]=X, \quad [\xi_\alpha,\xi_\beta]=0
\]
for any $\alpha,\beta\in\{1,2\}$.
Let us consider the left-invariant $1$-forms $\eta^1$ and $\eta^2$  determined by the dual $1$-forms of $\imath E_{11}$ and $-\imath E_{22}$, respectively,
and the left-invariant tensor field $\varphi$ such that $\varphi(X)=Y$, $\varphi(Y)=-X$ and $\varphi(\xi_1)=\varphi(\xi_2)=0$.
The manifold $U(2)$ is compact, connected, with Euler number $\chi (U(2))=0$, thus we can define a left-invariant Lorentz metric $g$ such that the vector fields $\xi_1$, $\xi_2$, $X$ and $Y$ form an orthonormal basis with $g(\xi_1,\xi_1)=-1$. Such a structure on $U(2)$ is constructed in the Riemannian context (\cite{DK}) and then it is adapted to the Lorentzian case (\cite{LP2}).

This structure is a normal indefinite $g.f.f$-structure and its associated Sasaki $2$-form $\Phi$ verifies $\Phi=d\eta^\alpha$, for any $\alpha\in\{1,2\}$, so that it turns out to be a Lorentz $\mathcal{S}$-structure on $U(2)$. Moreover, one sees at once that $U(2)$ has constant ${\varphi}$-sectional curvature $4$. We see that $U(2)$ does not verify the null Osserman condition with respect to $(\xi_1)_p$, for any $p\in U(2)$. In fact, fixing $p\in U(2)$ and putting
  \[
 u_1=X_p+(\xi_1)_p,\ u_2=Y_p+(\xi_1)_p,\ u_3=(\xi_2)_p+(\xi_1)_p,
  \]
  one has $u_1,u_2,u_3\in N((\xi_1)_p)$.
By (\ref{equivalente}), we have 
\begin{align*}
	{R}(Y_p,u_1)u_1&=Y_p+3g(Y_p,\varphi u_1)\varphi u_1
+\widetilde{\eta}(u_1) \widetilde{\eta}(u_1)Y_p=5Y_p,\\	
{R}((\xi_2)_p,u_1)u_1&= \sum_{\alpha=1}^2 (\xi_\alpha)_p + X_p=(\xi_2)_p+u_1.
\end{align*}
Analogously, for $u_2$, we obtain
\begin{align*}
	{R}(X_p,u_2)u_2&=X_p+3X_p +X_p=5X_p,\\	
	{R}((\xi_2)_p,u_2)u_2&=\sum_{\alpha=1}^2 (\xi_\alpha)_p +Y_p=(\xi_2)_p +u_2.
\end{align*}
For any $z\in u_3^\bot$, we have
\begin{align*}
	{R}(z,u_3)u_3&=-\widetilde{\eta}(u_3) \widetilde{\eta}(u_3){\varphi}^2z=0,\nonumber
\end{align*}
since $\widetilde{\eta}(u_3)=0$.

Then it is evident that the eigenvalues of $\bar{R}_{u_1}$ and $\bar{R}_{u_2}$ are $5$ and $1$ whereas $\bar{R}_{u_3}=0$.

\section{The $\varphi$-Null Osserman Condition}
In this section, inspired by the example of $U(2)$, we introduce a new Osserman condition that will be applied to Lorentz $g.f.f$-manifolds.

Let $(M,\varphi,\xi_\alpha,\eta^\alpha,g)$, $\alpha \in \{1,\ldots,s\}$, be a Lorentz $g.f.f$-manifold, it is easy to check that the timelike vector field must be a characteristic vector field. Without loss of generality we can assume that $\xi_1$ is the timelike vector field. 

Taking in mind the example in Section 3, we claim that, if $s\geq 2$, then the flatness of $\ker \varphi$ influences the behaviour of the Jacobi operators $\bar{R}_{u_\alpha}$ with $u_\alpha=(\xi_1)_p+(\xi_\alpha)_p$, for any $\alpha\in\{2,\ldots,s\}$ and $p\in M$. Since the matter is related to the null vector $u_\alpha$, we give the following Osserman condition.

Given a point $p$ of $M$, the set 
\[
S_\varphi ((\xi_1)_p)=S((\xi_1)_p)\cap \operatorname{Im}\varphi_p,
\]
is called the \emph{$\varphi$-celestial sphere} of $(\xi_1)_p$ at $p$. We define the analogous of the null congruence set, called the $\varphi$-null congruence set, denoted by $N_\varphi((\xi_1)_p)$, putting
\[
N_\varphi((\xi_1)_p)=\{u\in T_p M \;|\; u=(\xi_1)_p+x,\; x\in S_\varphi((\xi_1)_p)\}.
\]

Now, we are ready to state the definition of $\varphi$-null Osserman condition with respect to the timelike vector $(\xi_1)_p$ at a point $p\in M$.

\begin{definition}
Let $(M,\varphi,\xi_\alpha,\eta^\alpha,g)$ be a Lorentz $g.f.f$-manifold, $\dim M=2n+s$, $n,s\geq 1$, with timelike vector field $\xi_1$ and consider $p\in M$. $M$ is called $\varphi$-null Osserman with respect to $(\xi_1)_p$ at a point $p\in M$ if the characteristic polynomial of $\bar{R}_u$ is independent of $u\in N_\varphi((\xi_1)_p)$, that is the eigenvalues of $\bar{R}_u$ are independent of $u\in N_\varphi((\xi_1)_p)$.
\end{definition}
\begin{remark}
If $(M,\varphi,\xi,\eta,g)$ is a Lorentz almost contact manifold, then it can be considered as a Lorentz $g.f.f$-manifold with $s=1$. Obviously one has $S((\xi)_p)=S_\varphi((\xi)_p)$ and $N((\xi)_p)=N_\varphi((\xi)_p)$, for any $p\in M$. It follows that the null Osserman condition with respect to $\xi_p$ at a point $p$ coincides with the $\varphi$-null Osserman condition at the same point.
\end{remark}

It is clear that $U(2)$ verifies the $\varphi$-null Osserman condition with respect to $(\xi_1)_p$ at a point $p\in U(2)$. In fact,
we consider an arbitrary unit vector $z$ of $\operatorname{Im}\varphi_p$ putting $z=a X_p+ b Y_p$.
Setting $u_4=z+(\xi_1)_p$, we have $u_4\in N_\varphi((\xi_1)_p)$ and
\[
	u_4^\bot=span\{X_p+a(\xi_1)_p,Y_p+b(\xi_1)_p,(\xi_2)_p\}=span\{\varphi u_4, u_4,(\xi_2)_p\}.
\]
Then, we get
\begin{align*}
	{R}(\varphi u_4,u_4)u_4&={\varphi}u_4+3\varphi u_4+{\varphi}u_4=5{\varphi}u_4,\\
	{R}((\xi_2)_p,u_4)u_4&= \sum_{\alpha=1}^2 (\xi_\alpha)_p-{\varphi}^2u_4= (\xi_2)_p+(\xi_1)_p+z=(\xi_2)_p+u_4.\nonumber
\end{align*}
It follows that, for any $u=z+(\xi_1)_p$ in $N_\varphi((\xi_1)_p)$ with $z\in \operatorname{Im}\varphi_p$ and $g(z,z)=1$, the eigenvalues of $\bar{R}_u$ are $5$ and $1$, hence the eigenvalues of $\bar R_u$ are independent of the choice of $u\in N_\varphi((\xi_1)_p)$.

Taking into account the classical definitions for the Osserman manifolds, we introduce the globally $\varphi$-null Osserman condition.
\begin{definition}
Let $(M,\varphi,\xi_\alpha,\eta^\alpha,g)$ be a Lorentz $g.f.f$-manifold, $\dim M=2n+s$, $n,s\geq 1$, with timelike vector field $\xi_1$. 
If $M$ is $\varphi$-null Osserman with respect to $(\xi_1)_p$, for any $p\in M$, and the characteristic polynomial of $\bar{R}_u$ is independent of the choice of $p\in M$, then $M$ is said to be globally $\varphi$-null Osserman with respect to $\xi_1$.
\end{definition}

Looking again at the example of $U(2)$ one can see at once that it is a globally $\varphi$-null Osserman manifold with rispect to $\xi_1$. In fact, it is clear that the eigenvalues of $\bar R_u$ are independent of the point $p$. 

In the next theorem we prove, more generally, that each Lorentz $\mathcal{S}$-space form satisfies the $\varphi$-null Osserman condition.

\begin{theorem}\label{S-form}
Let $(M,\varphi,\xi_\alpha,\eta^\alpha,g)$, $\dim M=2n+s$, be a Lorentz $\mathcal{S}$-manifold with $\xi_1$ timelike and constant $\varphi$-sectional curvature, Let $p\in M$. Then $M$ verifies the $\varphi$-null Osserman condition with respect to the timelike characteristic vector at a point $p$.
\end{theorem}

\begin{proof}
Let $p\in M$. Denoting by $c$ the $\varphi$-sectional curvature, (\ref{equivalente}) holds with $\varepsilon=s-2$. 

Let $u$ be a vector in $N_\varphi((\xi_1)_p)$, hence $u=(\xi_1)_p+x_1$ with $x_1\in S_\varphi ((\xi_1)_p)$, and consider $x\in u^\bot$. We have:
\begin{align}
	g(\varphi u,\varphi u)&= g(u,u)-\sum_{\alpha=1}^s \varepsilon_\alpha \eta^\alpha (u)\eta^\alpha (u)=\eta^1(u)\eta^1 (u)=1,\label{phiu}\\
	g(\varphi x,\varphi u)&= g(x,u)-\sum_{\alpha=1}^s \varepsilon_\alpha \eta^\alpha (x)\eta^\alpha (u)=\eta^1 (x).\label{phix} 
\end{align}
By (\ref{equivalente}), (\ref{phiu}) and (\ref{phix}) we compute $R(x,u,u,w)$ for any $w\in T_p M$, obtaining
\begin{align}
	R_p(x,u,u,w)&
=-\frac{c+3(s-2)}{4}\{g(\varphi x,\varphi w) -\eta^1(x)g(\varphi u,\varphi w)\} \label{curvlightlike}\\
&  \quad -\frac{3}{4}(c-s+2)g(x,\varphi u)
g(w,\varphi u) \nonumber\\
&\quad -\{\widetilde{\eta}(w) \widetilde{\eta}(x) +\widetilde{\eta}(w) \eta^1(x) + g(\varphi w,\varphi x) 
+  \widetilde{\eta}(x)g(\varphi w,\varphi u)\} .\nonumber
\end{align}

Now, being $\dim M=2n+s$, we consider $\{x_1, \varphi x_1, x_3, \ldots , x_{2n}\}$ as an orthonormal base of $\operatorname{Im} \varphi_p$, which determines the bases $\mathfrak{B}=\{u,\varphi x_1, (\xi_2)_p,\ldots,(\xi_s)_p, x_3, \ldots , x_{2n}\}$ of $u^\bot$ and $\overline{\mathfrak{B}}=\{\overline{\varphi x}_1,(\bar{\xi}_2)_p,\ldots,(\bar{\xi}_s)_p,\bar{x}_3, \ldots , \bar{x}_{2n}\}$ of $\bar{u}^\bot$. For brevity, we also denote them by $\mathfrak{B}=\{e_i\}_{1\leq i\leq m}$, $\bar{\mathfrak{B}}=\{\bar{e}_i\}_{1\leq i\leq m-1}$, being $m=2n+s-1$. In general, for any $x\in u^\bot$, one has
\begin{equation}\label{espres}
\bar{R}_u(\bar{x})=-\sum_{i=1}^{m-1}R_p(x,u,u,e_i)\bar{e}_i.
\end{equation}
By (\ref{curvlightlike}) and (\ref{espres}) we obtain
\begin{align*}
\bar{R}_u (\overline{\varphi x}_1)&=\{\frac{c+3(s-2)}{4}+\frac{3}{4}(c-s+2)\}\overline{\varphi x}_1 +\overline{\varphi x}_1=(c+1)\overline{\varphi x}_1,\\
\bar{R}_u (\bar{x}_j)&=\frac{c+3(s-2)}{4}\bar{x}_j+\bar{x}_j=\frac{c+3s-2}{4}\bar{x}_j,\quad \forall j \in \{2,\ldots,2n\} \\	
\bar{R}_u ((\bar{\xi}_\beta)_p)&= \sum_{\gamma=2}^s\widetilde{\eta}( (\xi_\beta)_p)\widetilde{\eta}((\xi_\gamma)_p)(\bar{\xi}_\gamma)_p=\sum_{\gamma=2}^s(\bar{\xi}_\gamma)_p, \quad \forall\beta\in\{2,\ldots,s\}.
\end{align*}
It follows that the representation matrix of $\bar{R}_u$ with respect to $\bar{\mathfrak{B}}$ is independent of the choice of $u\in N_\varphi((\xi_1)_p)$. In particular, it is easy to compute that the other eigenvalues are $0$ and $s-1$, having eigenvectors $\bar{x}_\alpha=(\bar{\xi}_2)_p-(\bar{\xi}_\alpha)_p$, $\alpha\in\{3,\ldots,s\}$, and $\bar{x}=\sum_{\beta=2}^s(\bar{\xi}_\beta)_p$, respectively. This completes the proof.
\end{proof}

By the above proof we note that, as for $U(2)$, each Lorentz $\mathcal{S}$-manifold $(M,\varphi,\xi_\alpha,\eta^\alpha,g)$, $\dim M=2n+s$, with constant $\varphi$-sectional curvature is globally $\varphi$-null Osserman with respect to $\xi_1$.

From now on, since the Osserman conditions are formulated pointwise, to simplify the notation we omit any reference to the point, there is no ambiguity.

\section{The $\varphi$-null Osserman condition on Lorentz $\mathcal{S}$-manifolds with additional assumptions}
In this section we proceed with the study of $\varphi$-null Osserman manifolds and we will find an expression for the curvature tensor field of a $\varphi$-null Osserman Lorentz $\mathcal{S}$-manifold with two characteristic vector fields, using a suitable expression for null vectors. An analogous statement can be found in different contexts (\cite{GKV}). In the first part of this section we collect the technical issues needed for the main result, which will be provided in the second subsection.

\subsection{Technical results}
In \cite{GSV} the authors have given the esplicit construction of a complex structure on a $(4m+2)$-dimensional globally Osserman manifold with exactly two distinct eigenvalues of the Jacobi operators with multiplicities $1$ and $4m$ (see also \cite{GKV}). We will use such a construction, adapting it when the manifold verifies the $\varphi$-null Osserman condition at a point.

Following \cite{GKV1,GKV}, we recall that if $(M,g)$ is a Lorentz manifold and $u$ is a null vector of $T_p M$ then a non-degenerate subspace $W\subset u^\bot$ such that $\dim W=\dim \bar{u}^\bot$ is called a \emph{geometric realization} of $\bar{u}^\bot$. Let $\pi|_{W}:(W,g)\rightarrow (\bar{u}^\bot,g)$ be an isometry where, to simplify, we use the same letter $g$ for non-degenerate metrics on $W$ and $\bar{u}^\bot$. A vector $x\in W$ is said to be a \emph{geometrically realized eigenvector} of $\bar{R}_u$ in $W$ corresponding to an eigenvalue $\lambda$ if $\pi|_W(x)=\bar{x}$ is an eigenvector of $\bar{R}_u$ with eigenvalue $\lambda$ (\cite{GKV}).

\begin{remark}\label{GSV}
Let $(M,\varphi,\xi_\alpha,\eta^\alpha,g)$ be a $(2n+s)$-dimensional $\varphi$-null Osserman Lorentz $\mathcal{S}$-manifold at a point $p\in M$ and $u\in N_\varphi(\xi_1)$. We suppose that the Jacobi operator $\bar{R}_u$, restricted to $u^\bot\cap\operatorname{Im}\varphi$, has exactly two eigenvalues, $c_1$ and $c_2$, with multiplicities $1$ and $2n-2$. 

Since $u=\xi_1+x$, $x\in S_\varphi(\xi_1)$, using (\ref{PropS}), it is easy to see that the eigenvalues and the eigenvectors of the Jacobi operator $\bar{R}_u$ are connected with those of $R_x|_{x^\bot\cap\operatorname{Im}\varphi}$. Namely, one can prove that $v\in x^\bot\cap\mathrm{Im}\varphi$ is an eigenvector of $R_x$ related to the eigenvalue $\lambda$ if and only if it is a geometrically realized eigenvector of $\bar{R}_u$ related to the eigenvalue $\lambda+1$ (\cite{BC}).

Now, let us fix $p\in M$ and, following \cite{GSV}, identify $S_\varphi(\xi_1)\cong S^{2n-1}$. For any $x\in S^{2n-1}$ consider the operator $R_x:x^\bot\cap \operatorname{Im}\varphi \rightarrow x^\bot\cap \operatorname{Im}\varphi$ and  the line bundle over the sphere $S^{2n-1}$, defined by the eigenspace corresponding to the eigenvalue $c_1-1$ of $R_x$. Since any line bundle over a sphere is trivial, we have a map $J:S_\varphi(\xi_1)\rightarrow S_\varphi(\xi_1)$ such that $Jx=v_x$ for any $x\in S_\varphi(\xi_1)$, where $v$ is a global unit section of the line bundle. To simplify the writing, we put $\lambda=c_1-1$ and $\mu=c_2-1$.  Then, with the following sequence of claims, we proceed along the same lines as the authors made in \cite{GSV}, which the reader is referred to for details.

\medskip
\noindent{\bf Claim (a). }The map $J$ satisfies $J^2(x)=-x$ and $J(-x)=-J(x)$ for any $x\in S_\varphi(\xi_1)$.
\medskip\\
Considered the $2$-plane $V_x=span\{x,Jx\}$, if $w$ is a unit vector in $V_x$ then there exists $\theta\in [0,2\pi[$ such that $w=\cos(\theta)x+\sin(\theta)Jx$. Defining $z(w)=-\sin(\theta)x+\cos(\theta)Jx$, one proves that $z(w)$ is eigenvector of $R_w$ corresponding to $\lambda$, then $z(w)=\pm Jw$. Using this last formula, it follows $J^2(x)=-x$ and $J(-x)=-J(x)$ for any $x\in S_\varphi(\xi_1)$.

\medskip
\noindent{\bf Claim (b). }$J:\operatorname{Im}\varphi\rightarrow \operatorname{Im}\varphi$ is linear.
\medskip\\
The map $J$ is extended to $\operatorname{Im}\varphi$ putting $J(ax)=aJ(x)$, where $a\in\mathbb{R}$. Assuming that $J(\cos(\theta)x+\sin(\theta)y)=\cos(\theta)Jx+\sin(\theta)Jy$ for all angles $\theta$ and any $x$, $y$ unit vectors such that $y\bot V_x$, we obtain $J(x'+y')=J(x')+J(y')$, for any $x'$ and $y'$ such that $y'\bot V_x'$, which implies the claim.

\medskip
\noindent{\bf Claim (c). }$J(\cos(\theta)x+\sin(\theta)y)=\cos(\theta)Jx+\sin(\theta)Jy$ for all angles $\theta$ and any $x$, $y$ unit vectors such that $y\bot V_x$.
\medskip\\
Let us define $J'=\pm J$ and consider $A_{\theta}= \cos(\theta)x+\sin(\theta)y$, $B_\theta=\cos(\theta)Jx+\sin(\theta)J'y$. Assuming that $B_\theta$ is an eigenvector of $R_{A_\theta}$, then one has $B_\theta=\pm J A_\theta$, for any angle $\theta$. For $\theta=0$ one has $B_\theta= J A_\theta$ then the plus sign occurs. For $\theta=\frac{\pi}{2}$ it follows $J'y=B_\theta=JA_\theta=Jy$, i.e. $J'=J$, that implies the claim.

\medskip
\noindent{\bf Claim (d). }$R_{A_\theta}(B_\theta)=\lambda B_\theta$.
\medskip\\
Note that the claim is equivalent to proving that $R(B_\theta,A_\theta,A_\theta,B_\theta)=-\lambda$. Expanding this last formula in term of $x$, $Jx$, $y$ and $J'y$ one finds $R(Jx,x,y,J'y)+R(J'y,x,y,Jx)=\mu-\lambda$. After the following two technical lemmas one obtains the claim.

\begin{lemma}[\cite{GSV}]
\begin{itemize}
	\item [(1)] $R(z,v)w=-R(z,w)v$ when $v$, $w$ and $z$ are unit vectors such that $v\bot w$ and $w,v\bot V_z$.
	\item [(2)] $R(z,v)w=0$ when $v$, $w$ and $z$ are unit vectors such that $z\bot V_v$ and $z,v\bot V_w$.
	\item [(3)] $2R(x,y,J'y,Jx)=R(Jx,x,y,J'y)$.
	\item [(4)] $2R(J'y,x,y,Jx)=R(Jx,x,y,J'y)$.
\end{itemize}
\end{lemma}

\begin{lemma}[\cite{GSV}]
The curvature tensor satisfies $R(Jx,x,y,J'y)=\pm\frac{2(\mu-\lambda)}{3}$.
\end{lemma}
\end{remark} 

Now we give some remarks about a null vector of a Lorentz $\mathcal{S}$-manifold with two characteristic vector fields and then we prove a lemma.
  
\begin{remark}
Let $(M,\varphi,\xi_\alpha,\eta^\alpha,g)$, $\alpha\in\{1,2\}$, be a Lorentz $\mathcal{S}$-manifold with timelike vector field $\xi_1$ and $u$ a null vector in $T_p M$, $p\in M$. Since $TM=\operatorname{Im }\varphi\oplus \ker \varphi$, one can write $u$ in the following way
\[
u=\lambda x + a \,\xi_1 + b\, \xi_2,
\]
where $x\in\operatorname{Im}\varphi$ with $g(x,x)=1$. Being $u$ a null vector, we have $\lambda^2+ b^2 =a^2$ therefore there exists $\theta\in [0,2\pi[$ such that $u$ can be written as follows
\[
u=a(\cos\theta\: x+\xi_1+\sin \theta \:\xi_2),
\]
and it is not a restriction to use
\begin{equation}\label{ort}
u=\cos\theta \: x+\xi_1+\sin \theta \: \xi_2.
\end{equation}
For $\cos\theta\neq 0$ consider the vector $w=\tan \theta\: \xi_1 +\frac{1}{\cos\theta}\xi_2$. It is easy to check that $w$ is a unit vector orthogonal to $u$, therefore
\[
u^\bot=span\{u,\varphi x, x_2, \varphi x_2,\ldots x_n, \varphi x_n,w\}.
\]
Any $y\in u^\bot$ can be written as
\begin{equation}\label{ort01}
y=\rho u+\nu y'+\kappa w,
\end{equation}
where $y'\in span\{\varphi x, x_2, \varphi x_2,\ldots x_n, \varphi x_n\}\subset \operatorname{Im}\varphi_p\cap u^\bot$ and $\rho,\kappa,\nu\in\mathbb{R}$.
\end{remark}

We need to define two $(1,3)$-type tensors $S^*$ and $S_*$ putting
\begin{align*}
S^*(x,y)v&=\widetilde{\eta}(y)\widetilde{\eta}(v)x-\widetilde{\eta}(x)\widetilde{\eta}(v)y+g(y,v)\widetilde{\eta}(x)\widetilde{\xi}-g(x,v)\widetilde{\eta}(y)\widetilde{\xi},\\
S_*(x,y)v&=- g(\varphi y,\varphi v)\varphi^2 x+ g(\varphi x, \varphi v)\varphi^2 y.
\end{align*}
\begin{remark}\label{R1}
If $u\in N_\varphi(\xi_1)$ and $y\in \operatorname{Im}\varphi\cap u^\bot$, then $$g(S^*(u,y)u,y)-g(S_*(u,y)u,y)=0.$$
\end{remark}
The following lemma allows to state the expression of a curvature-like map $F$ when $F$ vanishes on a particular type of degenerate $2$-plane and it has a suitable behaviour with respect to the characteristic vector fields. 

\begin{lemma}\label{alcuni2pianidegeneri}
Let $(M,\varphi,\xi_\alpha,\eta^\alpha,g)$, $\alpha\in\{1,2\}$, be a Lorentz $g.f.f$-manifold with timelike vector field $\xi_1$. Fixed a point $p\in M$, let $F:(T_p M)^4\rightarrow \mathbb{R}$ be a curvature-like map such that, for any $x,y,v\in \operatorname{Im}\varphi$ and any $\alpha,\beta,\gamma\in\{1,2\}$, 
\begin{equation}\label{condF}
F(x,\xi_\alpha,y,v)=0,\	F(\xi_\alpha,x,\xi_\beta,y)=\varepsilon_\alpha\varepsilon_\beta g(x,y),\
	F(\xi_\alpha,x,\xi_\beta,\xi_\gamma)=0,\
	F(\xi_1,\xi_2,\xi_1,\xi_2)=0.
\end{equation}
Then the following statements are equivalent.
\begin{itemize}
	\item [a)] $F$ vanishes on any degenerate $2$-plane $\pi=span\{u,y\}$, with $u\in N_\varphi (\xi_1)$ and $y\in u^\bot \cap \operatorname{Im}\varphi$, 
	\item [b)] $F(x,y,v,z)=g(S_*(x,y)v,z)-g(S^*(x,y)v,z)$.
\end{itemize}
\end{lemma} 

\begin{proof}
An easy computation, using Remark \ref{R1}, shows that $b)\Rightarrow a)$. 

Conversely, fix $p\in M$ and consider the curvature-like map $H$ such that, for any $x,y,z,v\in T_p M$, 
\begin{equation}\label{H}
H(x,y,v,z)=F(x,y,z,w)-g(S_*(x,y)v,z)+g(S^*(x,y)v,z).
\end{equation}
Condition $a)$ and Remark \ref{R1} imply that $H$ vanishes on any degenerate $2$-plane $span\{u,y\}$, for any $u\in N_\varphi(\xi_1)$ and $y\in u^\bot\cap\operatorname{Im}\varphi$.
We start proving that $H$ vanishes on any degenerate $2$-plane. 
To see this, let $u$ be a null vector of $T_p M$ as in (\ref{ort}) such that $\cos\theta\neq 0$. By the hypotheses and using (\ref{ort01}), for any $y\in u^\bot$ we have
\begin{align*}
g(S_*(u,y)u,y)&=\left(\rho g(\varphi u,\varphi u)+\nu g(\varphi u, \varphi y')\right)^2-g(\varphi u,\varphi u)\left(\rho^2g(\varphi u,\varphi u)+\nu^2 g(\varphi y',\varphi y')\right)\\
&=\rho^2 g(\varphi u,\varphi u)^2-\rho^2g(\varphi u,\varphi u)^2-\nu^2g(\varphi y',\varphi y')g(\varphi u,\varphi u)=-\nu^2g( y', y')g(\varphi u,\varphi u),\\
g(S^*(u,y)u,y)&=-\widetilde{\eta}(u)\widetilde{\eta}(u)g(y,y),\\
F(u,y,u,y)&=\nu^2F(u,y',u,y')+2\kappa \nu F(u,y',u,w)+\kappa^2F(u,w,u,w)=\nu^2\cos^2\theta F(x,y',x,y')\\
\quad &+(1-\sin \theta)^2(\nu^2g(y',y')+ \kappa^2) =\nu^2g(\varphi u, \varphi u)F(x,y',x,y')+\widetilde{\eta}(u)\widetilde{\eta}(u)g(y,y)\\
&=\nu^2g(\varphi u, \varphi u)F(u',y',u',y')-\nu^2g(\varphi u, \varphi u)g(y',y')+\widetilde{\eta}(u)\widetilde{\eta}(u)g(y,y),
\end{align*}
where $u'=x+\xi_1$ which belongs to $N_\varphi (\xi_1)$. Hence one obtains 
\begin{equation}\label{H1}
H(u,y,u,y)=\nu g(\varphi u, \varphi u)F(u',y',u',y'),
\end{equation}
with $u'=x+\xi_1$ and $y\in u^\bot \cap \operatorname{Im}\varphi$. 

If $\cos\theta=0$, then $u=\xi_1\pm \xi_2$ and $u^\bot=span\{u\}\oplus \operatorname{Im}\varphi$. By direct computation, it is easy to verify that
\begin{equation}\label{H2}
H(u,y,u,y)=0,
\end{equation}
for any $y\in u^\bot$.

Equations (\ref{H1}) and (\ref{H2}) clearly imply that $H$ vanishes on any degenerate $2$-plane. Applying Lemma \ref{2pianidegeneri} to $H$ one has
\begin{equation}\label{01}
F(x,y,v,z)=k \left(g(x,v)g(y,z)-g(y,v)g(x,z)\right)+g(S_*(x,y)v,z)-g(S^*(x,y)v,z).
\end{equation}
By definition of $k$, using the hypotheses and (\ref{H}), we deduce $$k=\frac{H(\xi_\alpha,x,\xi_\alpha,x)}{\varepsilon_\alpha g(x,x)}=\frac{F(\xi_\alpha,x,\xi_\alpha,x)-g(x,x)}{\varepsilon_\alpha g(x,x)}=0.$$ 
Then, substituting in (\ref{01}), we obtain our assertion. \end{proof}

\subsection{Main results}
Now, we consider the following two standard tensor fields of type $(1,3)$, evaluating them at the point $p$:
\begin{align*}
R^0(x,y)v&=g(\pi^I(y),\pi^I(v))\pi^I(x)-g(\pi^I(x),\pi^I(v))\pi^I(y),\\
R^J(x,y)v&=g\left(J( \pi^I(y)),\pi^I(v)\right)J (\pi^I(x))-g\left(J (\pi^I(x)),\pi^I(v)\right)J( \pi^I(y))\\
&\quad +2 g\left(\pi^I(x),J (\pi^I(y))\right)J( \pi^I(v)),
\end{align*}
where $\pi^I:T_pM\rightarrow \operatorname{Im}\varphi$ is the projection on $\operatorname{Im}\varphi$ and $J$ is an almost Hermitian structure on $\operatorname{Im}\varphi$.

It is useful to note that $R^J$ and $R^0$ vanish on the triplets containing a characteristic vector and that they are orthogonal to $\xi_1$ and $\xi_2$, for any $x,y,v\in T_p M$.

Now we are ready to prove the following result.

\begin{theorem}
Let $(M,\varphi,\xi_\alpha,\eta^\alpha,g)$, $\alpha\in\{1,2\}$ and $n>1$, be a $(2n+2)$-dimensional Lorentz $\mathcal{S}$-manifold with timelike vector field $\xi_1$. The following three statements are equivalent. 
\begin{itemize}
	\item [a)] $M$ is $\varphi$-null Osserman with respect to $\xi_1$ and for any $u\in N_\varphi (\xi_1)$ the Jacobi operator $\bar{R}_u|_{\operatorname{Im}\varphi\cap u^\bot}$ has exactly two distinct eigenvalues $c_1$ and $c_2$ with multiplicities $1$ and $2(n-1)$, respectively. 
	\item [b)] There exist an almost complex structure $J$ on $\operatorname{Im}\varphi_p$ and  $c_1, c_2\in \mathbb{R}$ such that, for any $x,y,v\in T_p M$, 
	\begin{equation*}%\label{b}
	R(x,y)v=S^*(x,y)v- S_*(x,y)v+c_2R^0(x,y)v+\frac{c_1-c_2}{3} R^J(x,y)v.
	\end{equation*}
	\item [c)]
	\begin{itemize}
	\item [1.] For any $v\in span\{\xi_1\}$, $x\in \xi_1^\bot$ we have $$R(x,v)v=(\eta^1(v))^2(x-\widetilde{\eta}(x)\xi_2).$$
	\item [2.] There exist an almost complex structure $J$ on $\operatorname{Im}\varphi_p$ and $c_1, c_2\in \mathbb{R}$ such that, for any $v,y,x\in \xi_1^\bot$, we have
\begin{align*}
R(x,y)v&=\eta^2(v)\left(\eta^2(y)x-\eta^2(x)y\right)+\left(g(y,v)\eta^2(x)-g(x,v)\eta^2(y)\right)\widetilde{\xi}\\
&\quad +g(\varphi y,\varphi v)\varphi^2 x -g(\varphi x,\varphi v)\varphi^2 y+c_2R^0(x,y)v+\frac{c_1-c_2}{3}R^J(x,y)v.
\end{align*}
\end{itemize}
\end{itemize}
\end{theorem}

\begin{proof}
We begin proving $a)\Rightarrow b)$. Under the assumption \emph{a)} by Remark \ref{GSV} we know that $\operatorname{Im}\varphi_p$ is endowed with an almost complex structure $J$ such that $Jx$ is an eigenvector of $\bar{R}_u$ related to the eigenvalue $c_1$. To prove \emph{b)}, we consider the curvature-like map $F$ on $T_pM$ given by
\begin{equation}\label{F}
F(x,y,v,z)=R(x,y,v,z)+\mu g(R^0(x,y)v, z)+ \tau g(R^J(x,y)v,z),
\end{equation}
where $\mu,\tau\in \mathbb{R}$.

We want to apply Lemma \ref{alcuni2pianidegeneri} to $F$. About the hypotheses of Lemma \ref{alcuni2pianidegeneri}, we see at once that $F$ satisfies (\ref{condF}) since $F=R$ if one of its four arguments is a characteristic vector and (\ref{PropS}) hold. Thus we must only compute $F(u,y,u,y)$, for any degenerate vector $u\in N_\varphi(\xi_1)$ and $y\in u^\bot \cap \operatorname{Im}\varphi$.  

Namely, considering a null vector $u\in N_\varphi(\xi_1)$ and a vector $y\in u^\bot \cap \operatorname{Im}\varphi$, we find the suitable values of $\mu$ and $\tau$ in $\mathbb{R}$ for which $F$ vanishes on degenerate $2$-plane $\pi=span\{u,y\}$.

Putting $y_1=J x_1\in u^\bot$, one computes  
\begin{align}
	F(y_1,u,u,y_1)&=-g(R(y_1,u)u,y_1)+\mu g(R^0(y_1,u)u,y_1)+\tau g(R^J(y_1,u)u,y_1)\label{y_1}\\
	&=-c_1+\mu+3\tau\nonumber.
\end{align}
Analogously, if $y_2$ and $y'_2$ are orthonormal eigenvectors of $\bar{R}_u$ with respect to the eigenvalue $c_2$, then we have
\begin{align}
	F(y_2,u,u,y_2)&=-g(R(y_2,u)u,y_2)+\mu g(R^0(y_2,u)u,y_2)+\tau g(R^J(y_2,u)u,y_2))\label{y_2}\\
	&=(-c_2+\mu),\nonumber \\	
	F(y_2,u,u,y'_2)&=-g(R(y_2,u)u,y'_2)+\mu g(R^0(y_2,u)u,y'_2)+\tau g(R^J(y_2,u)u,y'_2))=0,\label{yy_2}\\
	F(y_2,u,u,y_1)&=-g(R(y_2,u)u,y_1)+\mu g(R^0(y_2,u)u,y_1)+\tau g(R^J(y_2,u)u,y_1))=0.\label{y_1y_2}	
\end{align}

Now, imposing $F=0$, we get
\begin{equation}\label{sol}
\mu=c_2 \text{ and } \tau=\frac{c_1-c_2}{3}.
\end{equation}
So, since a vector $y$ of $u^\bot \cap \operatorname{Im}\varphi$ can be written as $y=a y_1+b_j y_2^j$, where $y_1$ and $y_2^j$ are eigenvectors of $\bar{R}_u$ in $u^\bot\cap\xi_1^\bot$ corresponding to $c_1$ and $c_2$, respectively. By (\ref{y_1}), (\ref{y_2}), (\ref{yy_2}) and (\ref{y_1y_2}) we have 
\begin{align*}%\label{im}
F(y,u,u,y)&=a^2F(y_1,u,u,y_1)+ab_j F(y_1,u,u,y_2^j)+ab_k F(y_2^k,u,u,y_1)\\
&\quad +b_kb_jF(y_2^k,u,u,y_2^j)=0.\nonumber
\end{align*}
Therefore, applying Lemma \ref{alcuni2pianidegeneri}, we obtain $F(x,y,v,z)=g(S_*(x,y)v,z)-g(S^*(x,y)v,z)$, for any $x,y,v,z\in T_p M$. Then, by (\ref{F}) and (\ref{sol}), we get
\begin{equation*}%\label{R4}
R(x,y,v,z)=g(S_*(x,y)v,z)-g(S^*(x,y)v,z)-c_2 g(R^0(x,y)v,z)-\frac{c_1-c_2}{3} g(R^J(x,y)v,z). 
\end{equation*}
Thus, one obtains
\begin{equation*}%\label{R3}
R(x,y)v=-S_*(x,y)v+S^*(x,y)v+c_2 R^0(x,y)v+\frac{c_1-c_2}{3}R^J(x,y)v. 
\end{equation*}

The proof $b)\Rightarrow c)$ is straightforward. In fact, for any $v\in span\{\xi_1\}$, $x\in \xi_1^\bot$, we have 
\begin{align*}
	R(x,v)v&= S^*(x,v)v=(\eta^1(v))^2(x+\widetilde{\eta}(x)\xi_1+\varepsilon_1\widetilde{\eta}(x)\widetilde{\xi})=(\eta^1(v))^2\left(x-\widetilde{\eta}(x)\xi_2\right),
\end{align*}
so obtaining \emph{c)1.}
 
For any $v,y,x\in\xi_1^\bot$, by \emph{b)} one gets
\begin{align*}
R(x,y)v&=\eta^2(y)\eta^2(v)x-\eta^2(x)\eta^2(v)y+(g(y,v)\eta^2(x)-g(x,v)\eta^2(y))\widetilde{\xi}\\
&\quad+(-S_*+c_2R^0+\frac{c_1-c_2}{3}R^J)(x,y)v,
\end{align*}
that is \emph{c)2.}

Finally, we prove $c)\Rightarrow a)$. Consider $u\in N(\xi_1)$, $u=\xi_1+x_1$ and put $y_1=J x_1$. One has
\begin{equation*}
R(y_1,u)u=R(y_1,\xi_1)\xi_1+R(y_1,x_1)\xi_1+R(y_1,\xi_1)x_1+R(y_1,x_1)x_1.
\end{equation*}
So, using \emph{c)1.} and \emph{c)2.}, we have
\[
R(y_1,\xi_1)\xi_1=y_1\quad \text{ and } \quad R(y_1,x_1)x_1=(c_1-1)y_1.
\]
By \emph{c)2.}, for any $v\in \xi_1^\bot$, it is clear that
\[
	g(R(y_1,x_1)\xi_1,v)=-g(R(y_1,x_1)v,\xi_1)=0,\quad g(R(y_1,\xi_1)x_1,v)=g(R(x_1,v)y_1,\xi_1)=0.
\]
On the other hand, if $v=\xi_1$, then
\[
	g(R(y_1,x_1)\xi_1,\xi_1)=0,\quad g(R(y_1,\xi_1)x_1,\xi_1)=-g(y_1,x_1)=0.
\]
Hence, $\bar{R}_u(\overline{y}_1)=c_1\overline{y}_1$.

Analogously, considering $y_2\in (span\{x_1,y_1\})^\bot\cap \operatorname{Im}\varphi$, then 
\begin{equation*}
R(y_2,u)u=R(y_2,\xi_1)\xi_1+R(y_2,x_1)\xi_1+R(y_2,\xi_1)x_1+R(y_2,x_1)x_1.
\end{equation*}
As for $y_1$, using \emph{c)}, it is easy to check that $R(x_1,v)y_2=0$ and $R(y_2,x_1)v=0$. Moreover, applying \emph{c)1.}, we get
$$R(y_2,\xi_1)\xi_1=y_2.$$
The relation \emph{c)2.} implies
\begin{align*}
	R(y_2,x_1)x_1&=(c_2-1)y_2.
\end{align*}
Therefore we have $\bar{R}_u(\overline{y}_2)=c_2\overline{y}_2$.

Finally, to prove the $\varphi$-null Osserman condition, we have to check that every eigenvalue does not depend on $u\in N_\varphi(\xi_1)$. In fact, by \emph{c)} we find
\begin{align*}
	R(\xi_2,\xi_1)\xi_1&=0, \\
  R(\xi_2,x_1)x_1&=g(x_1,x_1)\widetilde{\xi}=\xi_1+\xi_2.\nonumber
\end{align*}
It is easy to see that, for any  $v\in \xi_1^\bot $
\begin{align*}
g(R(\xi_2,\xi_1)x_1,v)+g(R(\xi_2,x_1)\xi_1,v)&=-2g(R(\xi_2,x_1)v,\xi_1)+g(R(\xi_2,v)x_1,\xi_1)\\
	&=2g(x_1,v)-g(x_1,v)=g(x_1,v).
\end{align*}
Moreover, since
\[	
g(R(\xi_2,\xi_1)x_1,\xi_1)+g(R(\xi_2,x_1)\xi_1,\xi_1)=-g(R(\xi_2,\xi_1)\xi_1,x_1)=0,
\]
one obtains $R(\xi_2,\xi_1)x_1+R(\xi_2,x_1)\xi_1= x_1$. Then one gets $R(\xi_2,u)u=\xi_2+\xi_1+x_1=\xi_2+u$, so $\bar{R}_u(\overline{\xi}_2)=\overline{\xi}_2$. This proves \emph{a)}.

This concludes the proof. \end{proof}

\begin{remark}
Since $R$ has to satisfy the last formula in (\ref{PropS}), for any $x,y,v,z\in \operatorname{Im}\varphi$ one gets
\begin{equation}\label{link}
	(1-c_2)P(x,y;v,z)+\frac{c_1-c_2}{3}\left( g(R^J(x,y)\varphi v,z)+g(R^J(x,y)v,\varphi z)\right)=0.
\end{equation}
If $\varphi x_1$ is an eigenvector of $\bar{R}_u$, with $u=\xi_1+x_1\in N_\varphi(\xi_1)$, related to the eigenvalue $c_1$, then $\varphi=\pm J$ and (\ref{link}) yields $c_1-4c_2+3=0$. 
\end{remark}

By Theorem \ref{S-form}, it is a simple matter to prove the following result in the particular case of the Jacobi operator with exactly one eigenvalue.

\begin{proposition}
Let $(M,\varphi,\xi_\alpha,\eta^\alpha,g)$, $\alpha\in\{1,2\}$, $n>1$ be a $(2n+2)$-dimensional Lorentz $\mathcal{S}$-manifold with timelike vector field $\xi_1$. Then $M$ is $\varphi$-null Osserman with respect to $\xi_1$, and the Jacobi operator $\bar R_u|_{u^\bot\cap\operatorname{Im}\varphi}$ has a single eigenvalue $\lambda$, if and only if it is a Lorentz $\mathcal{S}$-space form with $\varphi$-sectional curvature $c=0$. Moreover, $\lambda=1$.
\end{proposition}

Now we end dealing with the case $n=1$, which is a special case because it is clear that any $4$-dimensional Lorentz $g.f.f$-manifold is $\varphi$-null Osserman with respect to $\xi_1$. More precisely, for any $u=\xi_1+x_1\in N_\varphi(\xi_1)$ the only eigenvector of the Jacobi operator $\bar{R}_u|_{u^\bot\cap\operatorname{Im}\varphi}$ is realized geometrically by $\varphi x_1$ in $u^\bot\cap\xi_1^\bot$. Unlike before, the eigenvalue of the Jacobi operator does not necessarily have to be one, as in the case of $U(2)$, but, when it is one, the $\varphi$-sectional curvature will be zero. About this case we have a non compact example. It is carried out by $\mathbb{R}^4$ endowed with the Lorentz $\mathcal{S}$-structure, constructed as follows (\cite{LP}). Denoting the standard coordinates with $\{x,y,z^{1},z^{2}\}$, we define on $\mathbb{R}^{4}$ two vector fields and two $1$-forms putting 
\[
\xi_{\alpha}=\frac{\partial}{\partial z^{\alpha}}, \quad \eta^{\alpha}=dz^{\alpha}+ydx,
\]
for any $\alpha\in\{1,2\}$. The tensor fields $\varphi$ and $g$ are given in the standard basis by
\begin{equation*}
F:=\left(
\begin{array}{cccc}
0&-1&0&0\\
1&0&0&0\\ 
0&y&0&0\\
0&y&0&0
\end{array}
\right)\qquad
G:=\left(
\begin{array}{cccc}
\frac{1}{2}&0&-y&y\\
0&\frac{1}{2}&0&0\\ 
-y&0&-1&0\\
y&0&0&1
\end{array}
\right)
\end{equation*}
respectively. It is easy to check that $(\mathbb{R}^{4},\varphi,\xi_{\alpha},\eta^{\alpha},g)$, $\alpha\in\{1,2\}$, is a Lorentz $\mathcal{S}$-manifold with different causal type of the characteristic vector fields. Moreover it is a Lorentz space form with $\varphi$-sectional curvature $c=0$. Therefore, by (\ref{equivalente}), one obtains
\begin{align*}
	{R}(X,Y,V)&=\widetilde{\eta}(X){g}({\varphi}V,{\varphi}Y)\sum_{\alpha=1}^2\xi_\alpha -\widetilde{\eta}(Y)
{g}({\varphi}V,{\varphi}X)\sum_{\alpha=1}^2\xi_\alpha-\widetilde{\eta}(Y) \widetilde{\eta}(V){\varphi}^2X\nonumber\\
&  \quad+\widetilde{\eta}(V)  \widetilde{\eta}(X) {\varphi}^2Y,\nonumber
\end{align*}
for any $X,Y,V\in \mathfrak{X}(\mathbb{R}^4)$.
Since $\operatorname{Im}\varphi=\left\langle X,Y\right\rangle$ where $X=\sqrt{2}(\frac{\partial}{\partial x}-y\xi_1-y\xi_2)$ and $Y=\sqrt{2}\frac{\partial}{\partial y}$, one has
\[
\bar{R}_u \varphi Z= \varphi Z, \quad \bar{R}_u \xi_2= \xi_2,
\]  
for any $Z=aX+bY$ and $u=\xi_1+ Z$ where $a^2+b^2=1$. Then the only eigenvalue of $\bar{R}_u$, $u\in N_\varphi(\xi_1)$, is $1$.

\end{document}